\documentclass[11pt,
]{amsart}

\usepackage{amssymb, color,cite, hyperref}

\usepackage{stmaryrd}

\usepackage{cancel}

\newtheorem{theorem}{Theorem}[section]
\newtheorem{lemma}[theorem]{Lemma}
\newtheorem{definition}[theorem]{Definition}

\newtheorem{corollary}[theorem]{Corollary}

\theoremstyle{definition}
\newtheorem{remark}[theorem]{Remark} 

\newcommand{\supp}{\operatorname{supp}}

\newcommand{\eps}{\varepsilon}

\newcommand{\bbR}{\mathbb R} \newcommand{\bbS}{\mathbb S}

\newcommand{\cC}{\mathcal C}

 \newcommand{\cL}{\mathcal L}
\newcommand{\cM}{\mathcal M} 
\newcommand{\cO}{\mathcal O}

\newcommand{\aver}[1]{\langle {#1} \rangle}

\title{Inverse Boundary Problem for the Two Photon Absorption Transport Equation}

\author[Plamen Stefanov]{Plamen Stefanov}
\email{stefanop@purdue.edu}  
\address{Department of Mathematics, Purdue University, West Lafayette, IN 47907}
\thanks{P.S.\ partially supported by the National Science Foundation under grant DMS-1900475.}

\author[Yimin Zhong]{Yimin Zhong}
\email{yimin.zhong@duke.edu}  
\address{Department of Mathematics, Duke University, Durham, NC 27710}

\date{}
\begin{document}

\begin{abstract}
We study the inverse boundary problem for the nonlinear two photon absorption radiative transport equation. We show that the absorption coefficients and  the scattering coefficient can be uniquely determined from the \emph{albedo} operator. If  the scattering is absent, we do not require smallness of the incoming source and the reconstruction of the absorption coefficients is  explicit.
\end{abstract}
\maketitle

\section{Introduction}
In this work we study the inverse boundary problem for the two photon absorption radiative transport equation. Two photon absorption happens when it takes two photons to excite a molecule from one state to another~\cite{rumi2010two, van1985two}. The probability of a two photon absorption at a given point is proportional to the light intensity there regardless of the incoming direction, which makes the corresponding term quadratic. One of the applications of two photon absorption is in medical imaging: the human body is not transparent to optical rays but it is more transparent to infrared ones. Then  fluorescent dyes with good two photon absorption rates can be used successfully with such a large wavelength excitation, see, e.g., \cite{Makarov:08, pawlicki2009two}. Other applications are pointed out in \cite{pawlicki2009two}; for example: microscopy, microfabrication, three-dimensional data storage, etc. For applications to photoacoustic imaging, we refer to \cite{bardsley2018quantitative} and the references there. 

Let $\Omega\subset\mathbb{R}^n$, $n\ge 2$ be an open bounded convex set with a $C^1$ boundary $\partial\Omega$, and let $\mathbb{S}^{n-1}$ be the unit sphere in $\mathbb{R}^n$, $\Gamma_{\pm} = \{(x, \theta)\in \partial\Omega\times\mathbb{S}^{n-1}\mid \pm\,n(x)\cdot \theta > 0 \}$, where $n(x)$ is the outer normal at $x\in\partial\Omega$. Denote by $u(x, \theta)$ the photon density function at spatial location $x\in\Omega$ in the direction $\theta\in\mathbb{S}^{n-1}$. Then our model is the following equation, see also \cite{Ren_Zhong2020},
\begin{equation}\label{EQ: TP RTE}
  \begin{aligned}
    \theta\cdot \nabla_x u(x, \theta)  + (\sigma_a(x, \theta) + \sigma_b(x, \theta) |\aver{u}|) u(x, \theta)  - Ku(x, \theta)&=0  \, &&\text{in } \Omega \times \mathbb{S}^{n-1}, \\
    u(x, \theta) &= f_{-}(x, \theta)\,& &\text{on } \Gamma_{-},
  \end{aligned}
\end{equation}
where $\aver{u}$ is the average of $u(x, \theta)$ over the angular variable $\theta$; that is, 
\begin{equation}\label{EQ: AVER}
  \aver{u} := \int_{\mathbb{S}^{n-1}} u(x, \theta) d\theta,
\end{equation}
with $d\theta$ being the normalized surface measure on $\mathbb{S}^{n-1}$.  When $u\ge0$, the absolute value in  $|\aver{u}|$ does not matter, of course but for general solutions, we include it to have a well-posed problem. 
The linear operator $ {K}$ is defined by
\begin{equation}\label{EQ: SCATTER}
  Ku(x, \theta) := \int_{\mathbb{S}^{n-1}} k(x, \theta', \theta) u(x, \theta') d\theta'.
\end{equation}
The coefficients $\sigma_a(x, \theta), k(x, \theta', \theta)$ are the usual total absorption and scattering coefficients, respectively. The coefficient $\sigma_b$ stands for strength of the nonlinear effect of two photon absorption, the term $\sigma_a + \sigma_b\aver{u}$ can be understood as the effective total absorption coefficient dependent on the solution. They are all assumed to be non-negative, and we impose smallness assumptions on $k$, $\sigma_b$ and $f_-$, see Definition~\ref{def_adm}.  

If the direct problem~\eqref{EQ: TP RTE} is uniquely solvable, one can define the usual \emph{albedo} operator 
\begin{equation}
  \mathcal{A}: f_{-}\mapsto f_{+},
\end{equation}
where $f_{+}(x, \theta) := u(x, \theta)|_{\Gamma_{+}}$ denotes the exiting photon density. This {albedo} operator is non-linear, and we are interested in finding out whether the albedo operator $\mathcal{A}$ determines uniquely the coefficients $\sigma_a(x, \theta), \sigma_b(x, \theta), k(x, \theta', \theta)$.

When $\sigma_b=0$, the equation~\eqref{EQ: TP RTE} is linear. Uniqueness and  recovery formulas for $\sigma_a$ and $k$, when $\sigma_a$ depends on $x$ only, were established in~\cite{Ch-St-Osaka} for $n\ge3$ and in \cite{SU-optical2D} for $n=2$ under a smallness assumption on $k$. The general case of $\sigma=\sigma(x,\theta)$ for $n\ge3$ was resolved in~\cite{ST-PAMC}. Stability estimates were proved in~\cite{Bal_J_stability, Bal_J_stab2}.  
Inverse radiative transport in the Riemannian setting was studied in \cite{MST_2011,SMT-gauge,SMT-stab,Steve2D,Steve2004,Sh-transport}, and for a different dynamical system, see~\cite{lai2019parameter}, there are also many other works regarding different types of boundary measurement, see~\cite{Bal-Monard-2012, Bal-Jol-2011,Bal-2009,Bal-K-M-2008,SU-APDE, zhao2019instability, lai2019inverse} and the references therein. References to earlier works can be found in the survey ~\cite{S-inside-out}. 

Inverse problems for non-linear versions of the transport equation (different from the one we study here) are studied in \cite{lai2021reconstruction, Klingenberg-Lai-Li-2021}. In~\cite{Ren_Zhong2020}, the authors considered the inverse medium problem under the same nonlinear model as~\eqref{EQ: TP RTE} and showed the uniqueness and stability of the reconstruction of absorption coefficients from internal data. 

The main result is the following. We show that we can recover $\sigma_a$, $k$, and $\sigma_b$ given the nonlinear operator $\mathcal{A}$. The idea of the proof is the following. 
If we take $f_-$  small, then we are in the linear regime and can use the result in \cite{Ch-St-Osaka} to recover $\sigma_a$ if it depends on $x$ only, and $k$. The latter requires $n\ge3$, see also \cite{SU-optical2D} for the 2D case. Next, we can take $f_-= f_0+ \delta f_1$, see \eqref{EQ:f}, with $0<\delta\ll1$ and $f_0>0$ smooth but $f_1$ singular in the $\theta$ variable only.  Then $f_0$ would not create singularities in solution at order $\mathcal{O}(\delta)$ but the effective absorption coefficient would involve $\sigma_b \aver{u_0}$, where $u_0$ is the leading $\mathcal{O}(1)$ term of the solution which is determined by $f_0$, see \eqref{EQ:21} and \eqref{EQ: U1}. This is the reason we require $f_0>0$, so that $\aver{u_0}>0$ and we can divide by it eventually to recover $\sigma_b$. Then choosing $f_1$ concentrated near a single $\theta'$ (and independent of $x$), allows us to reconstruct the X-ray transform of $\sigma_b$, and therefore $\sigma_b$ itself; see Theorem~\ref{THM: B}. 

Particularly, when $k=0$, one can solve the equation~\eqref{EQ: TP RTE} directly with $f_-$ in the form of $f_- = v_-(x)\delta_{\theta_0}(\theta)$ (a collimated source), see \eqref{2}, where $v_->0$ smooth. Then we are solving a Riccati ODE along each line $s\mapsto (x_0+ s\theta_0, \theta_0)$.   This allows us to recover $\sigma_a$ if it depends on $x$ only, and $\sigma_b$ through their attenuated X-ray transforms without the smallness assumption on $f_-$ (or of the perturbation of $f_-$ as in \eqref{EQ:f}), see Theorem~\ref{thm_sc-free}. This way, we may work with signals which are not necessarily small and will be less sensitive to additive background noise. 

The rest of the paper is organized as follows. In Section~\ref{SEC: PRE}, we state the preliminary results about the well posedness of the two photon absorption radiative transport model~\eqref{EQ: TP RTE}. Section~\ref{SEC: MAIN} consists of the main theorems about the reconstructions of the absorption and scattering coefficients, respectively. The scattering free case  $k = 0$ is discussed in Section~\ref{SEC: SCATTER FREE}. 

\section{Preliminaries}\label{SEC: PRE}

We first study the well posedness of \eqref{EQ: TP RTE} and of the {albedo} operator $\mathcal{A}$.  Define $\tau_{\pm}(x, \theta) :=\min\{ t\ge 0\mid x \pm t\theta \in\partial \Omega   \}$, which stands for the distance between $x$ and the boundary $\partial\Omega$ along $\pm\theta$. Set $\tau(x, \theta) = \tau_{-}(x, \theta) + \tau_{+}(\theta)$, and define the boundary measure $d\xi = |n(x)\cdot \theta| d\mu(x) d\theta$, where $d\mu(x)$ is the Lebesgue measure on $\partial \Omega$. Define the function space 
\[
\mathcal{H}^1(\Omega\times\mathbb{S}^{n-1}) := \left\{f\mid f\in L^1(\Omega\times\mathbb{S}^{n-1})\text{ and } \theta\cdot \nabla_x f \in L^1(\Omega\times\mathbb{S}^{n-1})  \right\}.
\]
 We further denote the function subspaces $L^1_{S}(\Gamma_{-}, d{\xi})\subset L^1(\Gamma_{-}, d{\xi})$ by 
\begin{equation}
  \begin{aligned}
    L^1_{S}(\Gamma_{-}, d{\xi}) &:= \Big\{ f \mid f\in L^1(\Gamma_{-}, d{\xi}) \text{ and } \|f\|_{\ast}<\infty\Big\},
  \end{aligned}
\end{equation} 
where the $\|\cdot\|_{\ast}$ norm is defined by 
\begin{equation}
  \|f\|_{\ast} := \Big\|\int_{\mathbb{S}^{n-1}} |f(x - \tau_{-}(x,\theta)\theta, \theta)| d\theta \Big\|_{L^{\infty}(\Omega)} .
\end{equation}

\begin{definition} \label{def_adm}
We call the tuple of functions $(\sigma_a, \sigma_b, k, f_{-})$  \emph{admissible} if 
\begin{enumerate}
  \item $\sigma_a, \sigma_b \in L^{\infty}(\Omega\times\mathbb{S}^{n-1})$, $\sigma_a \ge 0$ and $\sigma_b\ge 0$,
  \item $0\le k(x, \theta',\theta) \in L^{\infty}(\Omega\times \mathbb{S}^{n-1}\times \mathbb{S}^{n-1})$ and there exists a constant $\mu \in [0,1)$ such that 
$$
\|\tau\|_{L^{\infty}(\Omega\times\bbS^{n-1})} \| k\|_{L^{\infty}(\Omega\times 
  \bbS^{n-1} \times \bbS^{n-1})}\le\mu,
$$
  \item $f_{-}\in L_S^1(\Gamma_{-},d{\xi})$ and there exists $\nu\in [0, 1)$ such that
  \begin{equation}\nonumber
    \|\tau\|_{L^{\infty}(\Omega\times\bbS^{n-1})}\|\sigma_b\|_{L^{\infty}(\Omega\times\bbS^{n-1})}  \|f_{-}\|_{\ast} \le \nu (1-\mu)^2.
  \end{equation}
\end{enumerate}
\end{definition}

\begin{definition}
  We define the following operators: 
\[
Tu := -\theta\cdot \nabla_x u,\quad Su := \aver{u}, \quad \Sigma(m) u := -(\sigma_a + \sigma_b m) u.
\]
Then $\Sigma(m) = -(\sigma_a + \sigma_b m)$. Let $J(m) : L^1_S(\Gamma_{-}, d\xi) \mapsto {L^1(\Omega\times\bbS^{n-1})}$ be defined by 
  \begin{equation}
    J(m) f_{-}(x, \theta) =f_{-}(x - \tau_{-}(x, \theta)\theta, \theta) \exp\left( \int_0^{\tau_{-}(x, \theta)}\Sigma(m)(x - l\theta, \theta) dl  \right) ,
  \end{equation}
  and let $H(m): L^1(\Omega\times\bbS^{n-1})\to L^1(\Omega\times\bbS^{n-1})$ be defined by 
  \begin{equation}
      H(m) u(x, \theta) = \int_{0}^{\tau_{-}(x, \theta)} \exp\left(\int_0^l \Sigma(m) (x - s\theta,\theta) ds \right) Ku(x-l\theta, \theta)dl.
  \end{equation}
\end{definition}
\begin{lemma}\label{LEM: H}
    If the coefficients are admissible, then for any $m\in L^{\infty}(\Omega)$, 
\[
|H(|m|)u(x,\theta)| \le \mu \|\aver{|u|}\|_{L^{\infty}(\Omega)}.
\]
\end{lemma}
\begin{proof}
    Since $K |u|(x, \theta) \le \|k\|_{L^{\infty}(\Omega\times\bbS^{n-1}\times\bbS^{n-1})} \aver{|u|}(x)$, we  can derive
    \begin{equation}\nonumber
        \begin{aligned}
          |H(|m|) u(x, \theta)| &=\left|\int_{0}^{\tau_{-}(x, \theta)} \exp\left(\int_0^l \Sigma(|m|) (x - s\theta,\theta) ds \right) Ku(x-l\theta, \theta)dl   \right| \\
            &\le  \int_{0}^{\tau_{-}(x, \theta)} \exp\left(\int_0^l \Sigma(|m|) (x - s\theta,\theta) ds \right) K|u|(x-l\theta, \theta)dl\\
            &\le \int_{0}^{\tau_{-}(x, \theta)}  \|k\|_{L^{\infty}(\Omega\times\bbS^{n-1}\times\bbS^{n-1})}   \aver{|u|}(x-l
            \theta)dl \\
            &\le \|\tau\|_{L^{\infty} (\Omega\times\bbS^{n-1})} \|k\|_{L^{\infty}(\Omega\times\bbS^{n-1}\times\bbS^{n-1})}  \|\aver{|u|}\|_{L^{\infty}(\Omega)}\\
            &\le \mu  \|\aver{|u|}\|_{L^{\infty}(\Omega)}.
        \end{aligned}
    \end{equation}
\end{proof}

In particular, this shows that the  operator $H(|m|)$ is a contraction in $L^\infty(\Omega, L^1( \bbS^{n-1}))$. 

\begin{lemma}\label{lemma1}
If $(\sigma_a, \sigma_b, k, f_{-})$ is {admissible} and if $m\in L^\infty(\Omega)$, then the linear initial value problem 
\begin{equation}\label{EQ: ITER RTE}
  \begin{aligned}
    \left(T + \Sigma(|m|)   + K \right) u  &= 0\, &\text{ in } &\Omega \times \mathbb{S}^{n-1}, \\
    u(x, \theta) &= f_{-}(x, \theta)\, &\text{ on }& \Gamma_{-}
  \end{aligned}
\end{equation}
has a unique solution $u\in L^\infty(\Omega, L^1( \bbS^{n-1}))\cap \mathcal{H}^1(\Omega\times\mathbb{S}^{n-1}) $, 
and this solution  satisfies 
\[
\|\aver{|u|}\|_{L^{\infty}(\Omega)}  \le  \frac{1}{1-\mu}   \|f_{-}\|_{\ast}.
\]
\end{lemma}
\begin{proof}
The solution $u$ to~\eqref{EQ: ITER RTE} satisfies
\begin{equation}\label{EQ: SOL}
  u(x, \theta) = H(|m|)u(x, \theta) + J(|m|)f_{-}(x , \theta),
\end{equation}
and vice-versa, every solution to \eqref{EQ: SOL} solves \eqref{EQ: ITER RTE} (in a weak sense). 
Take the absolute value on both sides of~\eqref{EQ: SOL} and apply the operator $S$ to get, for every  $x\in\Omega$, 
\begin{equation}
  \begin{aligned}
    \aver{|u|}(x) &\le \int_{\mathbb{S}^{n-1}} |H(|m|) u (x, \theta) |d\theta  + \int_{\mathbb{S}^{n-1}} |J(|m|)f_{-}(x, \theta) |d\theta \\
    &\le \mu \|\aver{|u|}\|_{L^{\infty}(\Omega)} +\int_{\bbS^{n-1}}  |J(|0|)f_{-}(x, \theta) |d\theta \\
    &\le \mu \|\aver{|u|}\|_{L^{\infty}(\Omega)} + \|f_{-}\|_{\ast},
  \end{aligned}
\end{equation}
where we  used the Lemma~\ref{LEM: H}. 
The supremum on the left-hand-side satisfies
\begin{equation}\label{EQ: um}
\|\aver{|u|}\|_{L^{\infty}(\Omega)}  \le  \mu \|\aver{|u|}\|_{L^{\infty}(\Omega)} + \|f_{-}\|_{\ast}.
\end{equation}
In particular, this shows that the  operator $ H(|m|)$ is a contraction in $L^\infty(\Omega, L^1( \bbS^{n-1}))$, and that $J(|m|)f_{-}(x , \theta)$ belongs to that space, 
thus \eqref{EQ: SOL} is solvable in $L^\infty(\Omega, L^1( \bbS^{n-1}))$. Moreover, it satisfies the estimate in the lemma by \eqref{EQ: um}. Then we can apply $T$ to \eqref{EQ: SOL} to conclude that $u\in \mathcal{H}^1$ and solves  \eqref{EQ: TP RTE} in strong sense.  
\end{proof}

\begin{corollary}\label{COR: cor1}
Under the assumptions of Lemma~\ref{lemma1}, the solution $u$ to~\eqref{EQ: ITER RTE} also satisfies 
\begin{equation}\nonumber
  \Big\|\int_{\bbS^{n-1}} \int_{0}^{\tau_{-}(x,\theta)} |u(x - s\theta, \theta)| ds d\theta\Big \|_{L^{\infty}(\Omega)}\le\frac{\|\tau\|_{L^{\infty}(\Omega\times\bbS^{n-1})}}{1-\mu}  \|f_{-}\|_{\ast}.
\end{equation}
\end{corollary}
\begin{proof}
Using the estimate from Lemma~\ref{LEM: H}, and the equation~\eqref{EQ: SOL}, $\forall s\in [0, \tau_{-}(x, \theta)]$, 
\begin{equation}
  \begin{aligned}
    |u(x-s\theta, \theta)|&\le \mu \|\aver{|u|}\|_{L^{\infty}(\Omega)} + |f_{-}(x - s\theta - \tau_{-}(x-s\theta, \theta)\theta, \theta|  \\
    &= \mu \|\aver{|u|}\|_{L^{\infty}(\Omega)} + |f_{-}(x - \tau_{-}(x, \theta)\theta, \theta)|.
  \end{aligned}
\end{equation}
Apply the integrals with respect to $s$ and $\theta$, we obtain
\begin{equation}
  \int_{\bbS^{n-1}} \int_{0}^{\tau_{-}(x,\theta)} |u(x - s\theta, \theta)| ds d\theta \le \tau_{-}(x, \theta) \left( \mu \|\aver{|u|}\|_{L^{\infty}(\Omega)} + \int_{\mathbb{S}^{n-1}} |f_{-}(x - \tau_{-}(x,\theta)\theta, \theta)| d\theta\right). 
\end{equation}
Then take the supremum on both sides and use the conclusion of Lemma~\ref{lemma1} to get 
\begin{equation}
  \begin{aligned}
    \Big\|\int_{\bbS^{n-1}} \int_{0}^{\tau_{-}(x,\theta)} |u(x - s\theta, \theta)| ds d\theta\Big\|_{L^{\infty}(\Omega)} 
  &\le \|\tau\|_{L^{\infty}(\Omega\times\bbS^{n-1})}\left( \mu \|\aver{|u|}\|_{L^{\infty}(\Omega)} + \|f_{-}\|_{\ast}\right)\\  
  &\le  \|\tau\|_{L^{\infty}(\Omega\times\bbS^{n-1})}\frac{1}{1-\mu} \left(  \|f_{-}\|_{\ast}\right). 
  \end{aligned}
\end{equation}
\end{proof}

\begin{lemma}\label{LEM: UNIQUE}
  If $(\sigma_a, \sigma_b, k, f_{-})$ is {admissible}, then the radiative transport equation~\eqref{EQ: TP RTE} permits a unique solution $u(x, \theta)\in \mathcal{H}^1(\Omega\times\mathbb{S}^{n-1})$, and 
  \begin{equation}\label{EQ:lemma1}
      \|\aver{|u|}\|_{L^{\infty}(\Omega)} \le \frac{1}{1-\mu}\|f_{-}\|_{\ast}.
  \end{equation}
  In addition, if there exists a constant $c_0\ge0$ such that $f_{-}\ge c_0 $, then there is a constant $C = C(\Omega, \sigma_a, \sigma_b, k, f_{-}) > 0$ such that $u(x, \theta)\ge C c_0$. 
\end{lemma}

\begin{proof}
The proof is based on the Banach fixed point theorem.  Define the mapping $\mathcal{C}: L^{\infty}(\Omega)\mapsto L^{\infty}(\Omega)$ by $\mathcal{C} (m) := \aver{u}$, where $u(x, \theta)$ solves \eqref{EQ: ITER RTE}. 
Define the sets of functions $\mathcal{M}$ and $\cM_{+}$ by $$\mathcal{M}:= \Bigg\{ m \in L^{\infty}(\Omega) : |m(x)|\le \frac{1}{1-\mu}\|f_{-}\|_{\ast}\Bigg\},$$
and 
$$\mathcal{M}_{+}:= \Bigg\{ m \in L^{\infty}(\Omega) : 0\le m(x) \le \frac{1}{1-\mu}\|f_{-}\|_{\ast} \Bigg\}.$$
We prove that $\mathcal{C}$ is a contraction mapping on $\mathcal{M}$ (resp. $\cM_{+}$) with the $L^{\infty}(\Omega)$ metric. First we show that $\mathcal{C}:\cM\to \cM$. If $m\in \cM$, the solution to~\eqref{EQ: ITER RTE} will satisfy \eqref{EQ: SOL}. 
Take the absolute value on both sides of~\eqref{EQ: SOL} and apply the operator $S$.  By \eqref{EQ:lemma1},  $\aver{|u|}\in \mathcal{M}$, hence $\aver{u}\in \mathcal{M}$. 
When $f_{-}\ge 0$, from the theory of linear transport~\cite{dautray2012mathematical},
 the solution $u(x, \theta)$ to~\eqref{EQ: SOL} is non-negative as well through a fixed point iteration, thus we have the mapping $\mathcal{C}:\cM_{+}\to \cM_{+}$. 
In the next, we show  $\mathcal{C}$ is indeed a contraction mapping on both sets. Let $m_1, m_2\in\mathcal{M}$ (resp. $\cM_{+}$) and $u_1, u_2$ be the solutions to~\eqref{EQ: ITER RTE}, respectively. Denote $w = u_1 - u_2$, then
\begin{equation}
  \begin{aligned}
    \left( T + \Sigma(|m_1|) +K \right)w &= \sigma_bu_2(|m_1| - |m_2|), &&\text{in  $\Omega \times \mathbb{S}^{n-1}$}, \\
    w(x, \theta) &= 0\, &&\text{on $\Gamma_{-}$}.
  \end{aligned}
\end{equation}
Let $q(x, \theta):= \sigma_bu_2(|m_1| - |m_2|)$, then the solution $w(x, \theta)$ solves 
\begin{equation}
  w(x, \theta) = H(|m_1|) w(x, \theta) + \int_{0}^{\tau_{-}(x, \theta)} \exp\left(\int_{0}^l \Sigma(|m_1|)(x-s\theta) ds \right) q(x - l\theta, \theta) dl .
\end{equation}
Apply the integral operator $S$ on the second term on the right-hand-side to get
\begin{equation}\nonumber
  \begin{aligned}
   \bigg|\int_{\bbS^{n-1}}\int_{0}^{\tau_{-}(x, \theta)} & \exp\left(\int_{0}^l \Sigma(|m_1|)(x-s\theta) ds \right) q(x - l\theta, \theta) dl d\theta \bigg| \\
&\le \left\|\sigma_b (||m_1| - |m_2||)\right\|_{L^{\infty}(\Omega\times\mathbb{S}^{n-1})} 
    \bigg|\int_{\bbS^{n-1}}\int_{0}^{\tau_{-}(x, \theta)} |u_2(x - l\theta, \theta)| dl d\theta \bigg| \\&\le 
    \|\tau\|_{L^{\infty}(\Omega\times\bbS^{n-1})}  \left\|\sigma_b (||m_1| - |m_2||)\right\|_{L^{\infty}(\Omega\times\mathbb{S}^{n-1})} \frac{1}{1-\mu} \|f_{-}\|_{\ast}.
  \end{aligned}
\end{equation}
The last inequality comes directly from Corollary~\ref{COR: cor1}. As in the proof of Lemma~\ref{lemma1},
\begin{equation}\label{eq:15}
  \begin{aligned}
    \aver{|w|}(x) \le \mu \|\aver{|w|}\|_{L^{\infty}(\Omega)} +\frac{\|\tau\|_{L^{\infty}(\Omega\times\bbS^{n-1})}  \left\|\sigma_b |m_1 - m_2| \right\|_{L^{\infty}(\Omega\times\mathbb{S}^{n-1})}}{1-\mu} \|f_{-}\|_{\ast},
  \end{aligned}
\end{equation}
where we have used the triangle inequality $||m_1| -|m_2||\le |m_1 - m_2|$. Then use  $|\aver{w}(x) |\le \aver{|w|}(x)$ and $\aver{u_2}\in\cM$ (resp. $\cM_{+}$ when  $f_{-}(x, \theta)\ge 0$) to get
\begin{equation}\label{eq:16}
  \begin{aligned}
    |\aver{w}(x)|&\le \frac{\|\tau\|_{L^{\infty}(\Omega\times\bbS^{n-1})} \left\|\sigma_b |m_1 - m_2| \right\|_{L^{\infty}(\Omega\times\mathbb{S}^{n-1})} }{(1-\mu)^2} \|f_{-}\|_{\ast}.
  \end{aligned}
\end{equation}
By condition~(3) in Definition~\ref{def_adm},  $\cC$ is a contraction mapping on both $\cM$ and $\cM_{+}$ with the $L^{\infty}(\Omega)$ metric.  
Then by the Banach fixed point theorem, $\cC$ has a unique fixed point in $\cM$ (resp. $\cM_{+}$ when  $f_{-}(x, \theta)\ge 0$). Then \eqref{EQ:lemma1} follows  from  Lemma~\ref{lemma1}.  In particular, when $f_{-}(x,\theta)\ge c_0 > 0$, then $u(x, \theta)\ge 0$ and $0\le \aver{u}\le \frac{1}{1-\mu} \|f_{-}\|_{\ast}$,  therefore 
\begin{equation}
  \begin{aligned}
    u(x, \theta) &= H(\aver{u}) u(x, \theta) + J(\aver{u}) f_{-}(x, \theta) \ge J(\frac{1}{1-\mu} \|f_{-}\|_{\ast}) f_{-}(x, \theta) \\
    &\ge   c_0 \exp\left(-\textrm{diam}(\Omega) \left( \|\sigma_a\|_{L^{\infty}(\Omega\times \bbS^{n-1})} +  \frac{1}{1-\mu}\|\sigma_b\|_{L^{\infty}(\Omega\times\bbS^{n-1})} \|f_{-}\|_{\ast} \right) \right).
  \end{aligned}
\end{equation}
\end{proof}
\begin{remark}
  The mapping $\mathcal{C}$ may not be compact when $f_{-}\in L^1_S(\Gamma_{-}, d\xi)$,  therefore the Schauder fixed point theorem does not apply.
\end{remark}

\section{Main theorems}\label{SEC: MAIN}
In this section, we show that the nonlinear albedo operator determines the three  coefficients $\sigma_a$, $\sigma_b$, $k$, under the condition  $\sigma_a(x, \theta) = \sigma_a(x)$ and $\sigma_b(x, \theta) = \sigma_b(x)$. In the following, we consider a source function $f_{-}(x, \theta)$ in the form of 
\begin{equation}\label{EQ:f}
  f_{-}(x, \theta) = f_0(x, \theta) + \delta f_1(x, \theta)
\end{equation} 
with $\delta\to 0$  a scaling parameter, with $f_i\in L^1_S(\Gamma_{-}, 
 d {\xi})$  non-negative, $i=1,2$.  
Formally, the non-negative solution $u$ expands as 
\begin{equation}\label{EQ:21}
  u(x, \theta) = u_0(x, \theta) + \delta u_1(x, \theta) + \delta^2 u_2(x, \theta) +\cdots .
\end{equation}
Then $u_0$ and $u_1$ will satisfy the equations 
\begin{equation}\label{EQ: U0}
  \begin{aligned}
    ( T + \Sigma(\aver{u_0}) + K ) u_0 &= 0\, &\text{ in } &\Omega \times \mathbb{S}^{n-1}, \\
    u_0(x, \theta) &= f_0(x, \theta)\, &\text{ on }& \Gamma_{-}, 
  \end{aligned}
\end{equation}
and 
\begin{equation}\label{EQ: U1}
  \begin{aligned}
    ( T + \Sigma(\aver{u_0}) + K ) u_1 &= -\sigma_b \aver{u_1} u_0\, &\text{ in } &\Omega \times \mathbb{S}^{n-1}, \\
    u_1(x, \theta) &= f_1(x, \theta)\, &\text{ on }& \Gamma_{-}.
  \end{aligned}
\end{equation}
When the coefficients are admissible and $f_0=0$, then the equation~\eqref{EQ: U0} has unique solution $u_0 = 0$; and the equation~\eqref{EQ: U1} becomes the linear transport equation. Then one can follow the method in~\cite{Ch-St-Osaka} to decompose the singularities, which leads to the reconstruction of $\sigma_a$ and $k$,  the latter requires dimension $n\ge 3$.
 After the coefficients $\sigma_a$ and $k$ are recovered, we can select arbitrary nonzero $f_0\in L_S^1(\Gamma_{-}, d {\xi})$ such that $u_0$ is non-singular. Then in the equation~\eqref{EQ: U1}, the most singular part in the solution will come from the source $f_1$ if we select it to be singular in angular variable $\theta$. Therefore $\Sigma(\aver{u_0})$ can be recovered, and then $u_0$ can be solved from~\eqref{EQ: U0}, which finally reconstructs $\sigma_b$. In the following, we rigorously prove these claims.

\subsection{Reconstruction of $\sigma_a$} 
In next theorem, we show that we can recover the X-ray transform of $\sigma_a(x,\theta)$. As a corollary, if $\sigma_a$ is $\theta$-independent, one recovers it through the inverse X-ray transform~\cite{Novikov}. 

Here and below, we take sources approximating singular ones in the spirit of \cite{Ch-St-Osaka}. 
Let $B_1$ be the unit ball centered at origin in $\bbR^n$, $h\in C_0^{\infty}(B_1)$ with $0\le h\le 1$ and $h\equiv 1$ near origin be a cut-off function. Given $\theta'\in\bbS^{n-1}$, define the source function 
\begin{equation}\label{EQ:f_-}
f_{-}^{\eps, \delta}(x, \theta; \theta') =  \frac{\delta}{\omega_{n-1}\eps^{n-1}}h\left(\frac{\theta - \theta'}{\eps}\right),
\end{equation}
 where $\delta, \eps > 0$ are small parameters such that $f_{-}^{\eps, \delta}\in L_{-}^S(\Gamma_{-}, d\xi)$ and $\omega_{n-1}$ is the constant defined by
\begin{equation}\label{EQ: OMEGA}
   \omega_{n-1}: = \lim_{\eps\to 0} \int_{\bbS^{n-1}} \frac{1}{\eps^{n-1}} h\left(\frac{\theta - \theta'}{\eps}\right) d\theta.
\end{equation}
We view $f_{-}^{\eps, \delta}$ as $\delta$ times an approximation (a Friedrichs' mollifier) of the delta function $\delta_{\theta'}(\theta)$ on the sphere. Then $f_{-}^{\eps, \delta}$ plays the role of $\delta f_1$ in \eqref{EQ:f} with $f_0=0$ there.

\begin{theorem}
    Let $f_{-} = f_{-}^{\eps, \delta}$ and assume the tuple $(\sigma_a, \sigma_b, k, f_{-})$ is \emph{admissible}, then 
    \begin{equation}\nonumber
      \lim_{\gamma\to 0}\lim_{\eps,\delta\to 0} \int_{\mathbb{S}^{n-1}} \frac{u^{\eps,\delta}(x, \theta)}{\delta} h\left(\frac{\theta - \theta'}{\gamma}\right) d\theta = \exp\left(-\int_{0}^{\tau_{-}(x, \theta')}\sigma_a(x -s\theta', \theta')ds\right),
    \end{equation}
    where $u^{\eps,\delta}$ is the unique solution to~\eqref{EQ: TP RTE} with boundary condition $f_{-}^{\eps,\delta}$.
\end{theorem}

\begin{proof}
    Let $w^{\eps}$ be the unique solution to the following radiative transport equation:
\begin{equation}
  \begin{aligned}
    ( T + \Sigma(0) + K ) w^{\eps} &= 0 &\text{ in } &\Omega \times \mathbb{S}^{n-1}, \\
    w^{\eps}(x, \theta) &= \frac{1}{\omega_{n-1}\eps^{n-1}} h\left(\frac{\theta - \theta'}{\eps}\right) \, &\text{ on }& \Gamma_{-}.
  \end{aligned}
\end{equation}
The solution $w^{\eps}$ then satisfies  
\begin{equation}\label{EQ: W2}
    \begin{aligned}
        w^{\eps}(x, \theta) &= \frac{1}{\omega_{n-1}\eps^{n-1}} h\left(\frac{\theta - \theta'}{\eps}\right) \exp\left(-\int_{0}^{\tau_{-}(x, \theta)}\sigma_a(x -s\theta, \theta)ds\right) + H(0) w^{\eps}, 
    \end{aligned}
\end{equation}
where $|H(0) w^{\eps}|\le \mu \|\aver{w^{\eps}}\|_{L^{\infty}(\Omega)}$ which is uniformly bounded from Lemma~\ref{LEM: H}. Therefore, the following iterated limit holds
\begin{equation}\label{EQ: W}
  \begin{aligned}
    &\lim_{\gamma\to 0}\lim_{\eps\to 0} \int_{\bbS^{n-1}} w^{\eps}(x, \theta)  h\left(\frac{\theta - \theta'}{\gamma}\right)  d\theta \\
     = & \lim_{\gamma\to 0}\lim_{\eps\to 0} \int_{\bbS^{n-1}} \frac{1}{\omega_{n-1}\eps^{n-1}} h\left(\frac{\theta - \theta'}{\eps}\right) \exp\left(-\int_{0}^{\tau_{-}(x, {\theta})}\sigma_a(x -s{\theta}, {\theta})ds\right)   h\left(\frac{\theta - \theta'}{\gamma}\right)  d\theta 
    \\
    &\qquad + \lim_{\gamma\to 0}\lim_{\eps\to 0} \int_{\bbS^{n-1}} H(0) w^{\eps}(x, \theta)  h\left(\frac{\theta - \theta'}{\gamma}\right)  d\theta \\
    =&\exp\left(-\int_{0}^{\tau_{-}(x, \theta')}\sigma_a(x -s\theta', \theta')ds\right).
  \end{aligned}
\end{equation}
The term containing  $H(0)$  vanishes because when $\gamma\to 0$, 
\begin{equation}
  \begin{aligned}
    \left|\int_{\bbS^{n-1}} H(0) w^{\eps}(x, \theta)  h\left(\frac{\theta - \theta'}{\gamma}\right)  d\theta \right| &\le  \mu \|\aver{w^{\eps}}\|_{L^{\infty}(\Omega)}\int_{\bbS^{n-1}} h\left(\frac{\theta - \theta'}{\gamma}\right)  d\theta  \to 0.
  \end{aligned}
\end{equation}
Denote $\phi = \frac{1}{\delta} u^{\eps,\delta} - w^{\eps}$, then 
\begin{equation}
  \begin{aligned}
    (T + \Sigma(|\aver{u^{\eps,\delta}}|) + K)\phi&= \sigma_b |\aver{u^{\eps,\delta}}| w^{\eps} &\text{ in } &\Omega \times \mathbb{S}^{n-1}, \\
   \phi(x, \theta) &= 0\, &\text{ on }& \Gamma_{-}.
  \end{aligned}
\end{equation}
Then one can show that $\phi(x, \theta) = \cL_1(x, \theta) + \cL_2(x, \theta) $, where 
\begin{equation}\label{EQ: PHI}
  \begin{aligned}
    \cL_1 &= \int_0^{\tau_{-}(x, \theta)} \exp\left(\int_0^{l} \Sigma(|\aver{u^{\eps,\delta}}|)(x - s\theta, \theta) ds\right)  K \phi(x - l\theta, \theta) dl, \\ 
   \cL_2 &=-  \int_0^{\tau_{-}(x, \theta)} \exp\left(\int_0^{l} \Sigma(|\aver{u^{\eps,\delta}}|)(x - s\theta, \theta) ds\right)\sigma_b |\aver{u^{\eps,\delta}}|w^{\eps}(x - l\theta, \theta)   dl .
  \end{aligned}
\end{equation}
The first term $\cL_1$ is uniformly bounded in $L^{\infty}$ norm, this could be derived from the Lemma~\ref{LEM: H} and Lemma~\ref{LEM: UNIQUE} by observing that 
\begin{equation}
    \frac{1}{\eps^{n-1}}\int_{\mathbb{S}^{n-1}} h\left(\frac{\theta - \theta'}{\eps}\right) d\theta = \int_{\frac{1}{\eps}\bbS^{n-1}} h(\theta - \theta') d\theta \le c|\partial B_1|
\end{equation}
for some absolute constant $c > 0$. Therefore, 
\begin{equation}\label{EQ: L1}
 \int_{\bbS^{n-1}} \cL_1(x, \theta) h\left(\frac{\theta - \theta'}{\gamma}\right) d\theta  = \cO(\gamma^{n-1})\to 0, \text{ as } \gamma\to 0.
\end{equation}
For the second term $\cL_2$ we have 
\begin{equation}\nonumber
    \begin{aligned}
        \left|\int_{\bbS^{n-1}} \cL_2(x, \theta)  h\left(\frac{\theta - \theta'}{\gamma}\right) d\theta \right|&\le \int_{\bbS^{n-1}}\int_0^{\tau_{-}(x, \theta)}\sigma_b |\aver{u^{\eps,\delta}}|w^{\eps}(x - l\theta, \theta)   h\left(\frac{\theta - \theta'}{\gamma}\right) d\theta dl\\
        &\le \|\sigma_b\aver{u^{\eps,\delta}}\|_{L^{\infty}(\Omega)}  \int_{\bbS^{n-1}}\int_0^{\tau_{-}(x, \theta)} w^{\eps}(x - l\theta, \theta)   h\left(\frac{\theta - \theta'}{\gamma}\right) d\theta dl.
    \end{aligned}
\end{equation}
Note that $\|\sigma_b\aver{u^{\eps,\delta}}\|_{L^{\infty}(\Omega)}   = \cO(\delta)$ by Lemma~\ref{LEM: UNIQUE} and the integral part is uniformly bounded by the decomposition for $w^{\eps}$ in~\eqref{EQ: W2}, therefore 
\begin{equation}\label{EQ: L2}
    \lim_{\gamma\to 0} \lim_{\eps,\delta\to 0}\int_{\bbS^{n-1}} \cL_2(x, \theta)  h\left(\frac{\theta - \theta'}{\gamma}\right) d\theta  = 0.
\end{equation}
Combining~\eqref{EQ: W},~\eqref{EQ: L1} and~\eqref{EQ: L2}, we arrive at our conclusion.
\end{proof}

\subsection{Reconstruction of $k$} 
We show next that once $\sigma_a$ is known, one can recover $k$ pointwise.

When  $n\ge 3$, we 
let $\theta, \theta'\in\bbS^{n-1}$ such that $\theta\nparallel \theta'$ and  denote $\pi_{\theta,\theta'}(x)$ the projection of $x$ onto the subspace $\Theta$ spanned by $\theta, \theta'$. Let $\theta'_{\perp}\in \Theta = \text{span}(\theta,\theta')$ be the unit vector such that $\theta'_{\perp}\cdot \theta' = 0$. Take any $\varphi \in C_0^{\infty}(-1,1)$ that $0\le \varphi\le 1$ and $\int_{\bbR}\varphi(t) dt = 1$. We then define the test function 
\begin{equation}
    \phi_{\gamma_1, \gamma_2}(x, \theta, \theta') = \frac{1}{\gamma_1} \varphi\left(\frac{x\cdot \theta'_{\perp}}{\gamma_1\theta\cdot \theta'_{\perp}}\right) h\left(\frac{x - \pi_{\theta,\theta'}(x)}{\gamma_2}\right),
\end{equation}
we also define the source function $f_{-}^{\eps, \eps',\delta}$ in the form of
\begin{equation}
    f_{-}^{\eps,\eps', \delta}(x, \theta; x' ,\theta') = \frac{\delta}{\omega_{n-1}^2\eps^{n-1}}h\left(\frac{x - x'}{\eps'}\right) h\left(\frac{\theta - \theta'}{\eps}\right)
\end{equation}
such that $f_{-}^{\eps,\eps',\delta}\in L_{1}^S(\Gamma_{-}, d\xi)$, the constant $\omega_{n-1}$ is defined by~\eqref{EQ: OMEGA}.

\begin{theorem}
     Let $n\ge3$, set $f_{-} = f_{-}^{\eps, \eps',\delta}$, and assume the tuple $(\sigma_a, \sigma_b, k, f_{-})$ is admissible. Then
    \begin{equation}\nonumber
        \begin{aligned}
            &\lim_{\gamma_1\to 0}\lim_{\gamma_2\to 0} \lim_{\eps'\to 0}  \lim_{\eps\to 0} \lim_{\delta\to 0}\int_{\partial\Omega}\frac{u^{\eps,\eps', \delta}(x + \tau_{+}(x, \theta)\theta, \theta; x',\theta')}{\eps'^{n-1}\delta} \phi_{\gamma_1, \gamma_2}(x' - x + \tau_{-}(x, \theta')\theta', \theta, \theta') d\mu(x') \\
            &=  \exp\left(-\int_{0}^{\tau^{+}(x, \theta)} \sigma_a(x + s\theta)ds\right)\exp\left(-\int_{0}^{\tau^{-}(x, \theta')} \sigma_a(x - s\theta')ds\right) k(x, \theta', \theta),
        \end{aligned}
    \end{equation}
where $u^{\eps, \eps',\delta}(x, \theta;x',\theta')$ is the unique solution to~\eqref{EQ: TP RTE} with boundary condition $f_{-}^{\eps,\eps', \delta}$. The limit holds in $L^1_{\textrm{loc}}(\Omega\times ( \bbS^{n-1}\times\bbS^{n-1}\backslash D) )$ where $D = \{(\theta, \theta')\in\bbS^{n-1}\times \bbS^{n-1}\mid \theta\nparallel \theta'\}$.
\end{theorem}
\begin{proof}
Similar to the section 3 of~\cite{Ch-St-Osaka}, we can write the solution decomposed as 
\begin{equation}\label{EQ: DECOMP}
    \begin{aligned}
        u^{\eps, \eps', \delta}(x, \theta) &= J(|\aver{u^{\eps, \eps',\delta}}|)f_{-} +  H(|\aver{u^{\eps,\eps', \delta}}|) J(|\aver{u^{\eps, \delta}}|)f_{-} \\&\quad + (I - H(|\aver{u^{\eps, \eps',\delta}}|))^{-1} H^2(|\aver{u^{\eps, \eps',\delta}}|) J(|\aver{u^{\eps, \eps',\delta}}|)f_{-} \\
        &= \cL_1(x, \theta) + \cL_2(x, \theta) + \cL_3(x, \theta),     
    \end{aligned}
\end{equation}
with the terms there corresponding to the ballistic, the single-scattering, and the  multiple-scattering components. First, it is simple to see that when $\eps$ is small enough so that $|\theta - \theta'| > \eps$, then $h\left( \frac{\theta-\theta'}{\eps} \right) = 0$; hence 
\begin{equation}\label{EQ: K L1}
    \begin{aligned}
      \int_{\partial\Omega}  & \frac{\cL_1(x + \tau_{+}(x,\theta)\theta, \theta)}{\eps'^{n-1}\delta} \phi_{\gamma_1, \gamma_2}(x' - x + \tau_{-}(x, \theta')\theta', \theta, \theta') d\mu(x')     \\
         &= \int_{\partial\Omega} \frac{1}{\omega_{n-1}^2\eps'^{n-1}\eps^{n-1}} h\left(\frac{x - \tau_{-}(x, \theta)\theta - x'}{\eps}\right) h\left( \frac{\theta-\theta'}{\eps} \right) \\&\quad \times\exp\left(-\int_0^{\tau(x, \theta)} \Sigma(|\aver{u^{\eps,\delta}}|)(x - s\theta, \theta) ds\right)\phi_{\gamma_1, \gamma_2}(x' - x + \tau_{-}(x, \theta')\theta', \theta, \theta') d\mu(x')\\
         &=  0.
    \end{aligned}
\end{equation}
Next, we compute the contribution of the single-scattering term. Let $E(x, y, m)$ denote
\[        E(x, y, m) = \exp\left(|x- y|\int_{0}^1 \Sigma(m)(x + s(y-x))ds\right). \]
In order to make the derivation concise, we also introduce the following notation
\begin{equation}
    \begin{aligned}
        x_{\pm,\theta} &= x \pm \tau_{\pm}(x, \theta)\theta, \\
        y_{l,\theta} &= x_{+,\theta} - l\theta,\\
        z_{l,\theta, \theta''} & = y_{l,\theta}- \tau_{-}(y_l, \theta'')\theta''.
    \end{aligned}
\end{equation}
\begingroup
\allowdisplaybreaks
Then we could write
    \begin{align*}
        &\lim_{\eps'\to 0}\lim_{\eps\to 0}\lim_{\delta\to 0}\int_{\partial\Omega} \frac{\cL_2(x + \tau_{+}(x,\theta)\theta, \theta)}{\eps'^{n-1}\delta} \phi_{\gamma_1, \gamma_2}(x' - x + \tau_{-}(x, \theta')\theta', \theta, \theta') d\mu (x') \\ 
        =& \lim_{\eps'\to 0}\lim_{\eps\to 0}\lim_{\delta\to 0}\int_{\partial\Omega} \int_{0}^{\tau(x, \theta)} \int_{\bbS^{n-1}} E(x_{+,\theta}, y_{l,\theta}, |\aver{u^{\eps, \eps',\delta}}| )  E( y_{l,\theta},z_{l,\theta,\theta''}, |\aver{u^{\eps,\eps',\delta}}| )\times \\ &\quad\frac{1}{\omega_{n-1}\eps'^{n-1}}  h\left(\frac{z_{\theta''} - x'}{\eps'}\right) \frac{1}{\omega_{n-1}\eps^{n-1}}h\left(\frac{\theta'' - \theta'}{\eps}\right) k(y_{l,\theta}, \theta'', \theta) \times \\
        &\quad \phi_{\gamma_1, \gamma_2}(x' - x + \tau_{-}(x, \theta')\theta', \theta, \theta')  d\theta'' dld\mu (x') \\
        =& \lim_{\eps'\to 0}\lim_{\eps\to 0}\int_{\partial\Omega} \int_{0}^{\tau(x, \theta)} \int_{\bbS^{n-1}}  E(x_{+,\theta}, y_{l,\theta}, 0 ) E( y_{l,\theta},z_{l,\theta,\theta''}, 0 )\times \\ &\quad\frac{1}{\omega_{n-1}\eps'^{n-1}}h\left(\frac{z_{\theta''} - x'}{\eps'}\right) \frac{1}{\omega_{n-1}\eps^{n-1}}h\left(\frac{\theta'' - \theta'}{\eps}\right) k(y_{l,\theta}, \theta'', \theta)  \times \\
        &\quad \phi_{\gamma_1, \gamma_2}(x' - x + \tau_{-}(x, \theta')\theta', \theta, \theta')d\theta''dl d\mu (x') \\
        =& \lim_{\eps'\to 0}\int_{\partial\Omega} \int_{0}^{\tau(x, \theta)} E(x_{+,\theta}, y_{l,\theta}, 0 ) E( y_{l,\theta},z_{l,\theta,\theta'}, 0 )\times \\ &\quad\frac{1}{\omega_{n-1}\eps'^{n-1}}h\left(\frac{z_{\theta'} - x'}{\eps'}\right)k(y_{l,\theta}, \theta', \theta)\times \\
        &\quad \phi_{\gamma_1, \gamma_2}(x' - x + \tau_{-}(x, \theta')\theta', \theta, \theta') dl  d\mu (x') \\
        =& \int_{0}^{\tau(x, \theta)}E(x_{+,\theta}, y_{l,\theta}, 0 ) E( y_{l,\theta},z_{l,\theta,\theta'}, 0 ) k(y_{l,\theta}, \theta', \theta) \times \\&\quad \phi_{\gamma_1, \gamma_2}(y_{l,\theta} - x + \tau_{-}(x, \theta')\theta', \theta, \theta')dl.
    \end{align*}
\endgroup
The right-hand-side has the limit 
\begin{equation}
    \begin{aligned}
        \lim_{\gamma_1\to 0}\lim_{\gamma_2\to 0} &\int_{0}^{\tau(x, \theta)}E(x_{+,\theta}, y_{l,\theta}, 0 ) E( y_{l,\theta},z_{l,\theta,\theta'}, 0 ) k(y_{l,\theta}, \theta', \theta) \\&\quad \times \phi_{\gamma_1, \gamma_2}(y_{l,\theta} - x + \tau_{-}(x, \theta')\theta', \theta, \theta')dl \\  
        &= \lim_{\gamma_1\to 0}\int_{0}^{\tau(x, \theta)}E(x_{+,\theta}, y_{l,\theta}, 0 ) E( y_{l,\theta},z_{l,\theta,\theta'}, 0 ) k(y_{l,\theta}, \theta', \theta)\frac{1}{\gamma_1} \varphi\left(\frac{\tau^{+}(x, \theta) - l}{\gamma_1}\right) dl \\
        & = E(x_{+,\theta}, x, 0 ) E( x,x_{-,\theta'}, 0 ) k(x, \theta', \theta).
    \end{aligned}
\end{equation}
To show that the multi-scattering contribution is zero, we only need to show that $\frac{1}{(\eps')^{n-1}\delta}\cL_3(x, \theta)\in L^1(\Omega\times\bbS^{n-1})$ uniformly, hence uniform bounded in $L^1(\Gamma_{\pm}, d\xi)$. 
Given any $\chi\in C_0^{\infty}(\Omega \times ( \bbS^{d-1}\times \bbS^{d-1} \backslash D))$, we have
\begin{equation}\label{EQ: MULTI}
  \begin{aligned}
    &\int_{\Omega\times\bbS^{n-1}\times \Gamma_{-}}\frac{\cL_3(x + \tau_{+}(x, \theta), \theta; x',\theta')}{\eps'^{n-1}\delta} \phi_{\gamma_1,\gamma_2}(x'-x+\tau_{-}(x, \theta')\theta', \theta, \theta')  \chi  d\xi(x', \theta') d\mu(x) d\theta \\
    & \le \frac{1}{\gamma_1}\int_{T_{\gamma_2}} \frac{\cL_3(x + \tau_{+}(x, \theta), \theta; x',\theta')}{\eps'^{n-1}\delta}  \chi d\mu(x) d\theta d\xi(x', \theta'),
  \end{aligned}
\end{equation}
where $T_{\gamma_2} = \{(x, \theta, x' ,\theta')\in \Omega\times\bbS^{n-1}\times \Gamma_{-}\cap \supp\chi\text{ and } |x - x' - \pi_{\theta, \theta'}(x-x')|\le c \gamma_2\}$. When  $\frac{1}{(\eps')^{n-1}\delta}\cL_3(x, \theta)$ is uniformly bounded in $L^1(\Gamma_{+}, d\xi)$, the integrand of~\eqref{EQ: MULTI} is an $L^1$ function. On the other hand, $\textrm{meas}(T_{\gamma_2})\to 0$ as $\gamma_2\to 0$, therefore the integral vanishes as $\gamma_2\to 0$. In the following, we  prove $\frac{1}{(\eps')^{n-1}\delta}\cL_3(x, \theta)\in L^1(\Omega\times\bbS^{n-1})$ with a uniform bound there with respect to $\eps'\ll1$ and $\delta\ll1$.

Since $(I - H(|\aver{u^{\eps,\eps',\delta}}|))^{-1}$ is a uniformly bounded operator in $L^1(\Omega\times\bbS^{n-1})$, we merely  have to show that $ \frac{1}{(\eps')^{n-1}\delta}H^2(|\aver{u^{\eps,\eps',\delta}}|)J(|\aver{u^{\eps, \eps', \delta}}|) f_{-}$ is also uniformly bounded, see \eqref{EQ: DECOMP}. 
Let $y_{l,\theta} = x - l\theta$, $z_{s, \theta''} = y_{l,\theta} - s\theta''$, $w_{\theta'''} = z_{s,\theta''} - \tau_{-}(z_{s,\theta''},\theta''')\theta'''$.  Then 
\begin{equation}
    \begin{aligned}
        & \left|\frac{H^2(|\aver{u^{\eps, \eps',\delta}}|) J(|\aver{u^{\eps, \eps',\delta}}|)f_{-}(x, \theta)}{\eps'^{n-1}\delta}\right| \\
        &\le \int_{0}^{\tau_{-}(x,\theta)} \int_{\bbS^{n-1}}\int_{0}^{\tau_{-}(y_{l,\theta},\theta'')}\int_{\bbS^{n-1}} E(x, y_{l,\theta}, |\aver{u^{\eps, \eps',\delta}}|) E(y_{l,\theta}, z_{s,\theta''}, |\aver{u^{\eps, \eps',\delta}}|) \times \\&\quad E(z_{s,\theta''}, w_{\theta'''}, |\aver{u^{\eps, \eps',\delta}}|) k(y_{l,\theta}, \theta'',\theta) k(z_{s,\theta''}, \theta''', \theta'') |f_{-}(w_{\theta'''}, \theta''')| d\theta''' ds d\theta'' dl \\
        &\le \int_{0}^{\tau_{-}(x,\theta)} \int_{\bbS^{n-1}}\int_{0}^{\tau_{-}(y_{l,\theta},\theta'')}\int_{\bbS^{n-1}}  k(y_{l,\theta}, \theta'',\theta) k(z_{s,\theta''}, \theta''', \theta'') |f_{-}(w_{\theta'''}, \theta''')| d\theta''' ds d\theta'' dl .
    \end{aligned}
\end{equation}
Since $z_{s,\theta''} = y_{l,\theta} - s\theta'' $, we change the variable that $dz_{s, \theta''} = s^{n-1} ds d\theta''$, and recall the formula  
\begin{equation}
    \int_{\Omega\times \bbS^{n-1}} g(x, \theta) dx d\theta = \int_{\Gamma_{-}} \int_0^{\tau_+(x', \theta)} g(x' + t \theta, \theta) dt d\xi(x', \theta),
\end{equation}
see \cite{Ch-St-Osaka}, with $x' = x - \tau_{-}(x,\theta)\theta$. We obtain 
\begin{equation}
    \begin{aligned}
        &\int_{0}^{\tau_{-}(x,\theta)} \int_{\bbS^{n-1}}\int_{0}^{\tau_{-}(y_{l,\theta},\theta'')}\int_{\bbS^{n-1}}  k(y_{l,\theta}, \theta'',\theta) k(z_{s,\theta''}, \theta''', \theta'') |f_{-}(w_{\theta'''}, \theta''')| d\theta''' ds d\theta'' dl\\
        &=  \int_{0}^{\tau_{-}(x,\theta)}\int_{\Gamma_{-}}\int_{0}^{\tau_{+}(w_{\theta'''}, \theta''')} k(y_{l,\theta}, \theta'',\theta) k(w_{\theta'''} + t\theta''', \theta''', \theta'')\times \\&\quad \frac{1}{\eps'^{n-1}} |f_{-}(w_{\theta'''}, \theta''')|  s^{1-n}  dt d\xi(w_{\theta'''},\theta''') dl \\
        &\le C \left\|\frac{1}{\eps'^{n-1}}f_{-}\right\|_{L^1(\Gamma_{-}, d\xi)}\int_{0}^{\tau_{-}(x,\theta)} \int_{0}^{\tau_{+}(x', \theta')}    s^{1-n}dt  dl\in L^1(\Omega\times\bbS^{n-1}),
    \end{aligned}
\end{equation}
which is uniformly bounded in $L^1(\Omega\times\bbS^{n-1})$ with respect to $\eps'$, where $s = |y_{l,\theta} - (x' + t\theta')|$ and $C =\|k\|^2_{L^{\infty}(\Omega\times\bbS^{n-1}\times\bbS^{n-1})} $. 
\end{proof}
\subsection{Reconstruction of $\sigma_b$}
Let the source function $f_{-}$ be chosen in the following form
\begin{equation}
  f_{-}^{\eps, \delta} (x, \theta; \theta')= c_0 + \frac{\delta}{\omega_{n-1}\eps^{n-1}} h \left(\frac{\theta - \theta'}{\eps}\right),
\end{equation} 
where $c_0$ is a positive constant and $\delta, \eps$ are positive small  parameters.
Compared with \eqref{EQ:f_-}, here we have added $f_0=c_0$ in \eqref{EQ:f}. 

\begin{theorem}\label{THM: B}
  Let $f_{-} = f_{-}^{\eps, \delta}$ and assume the tuple $(\sigma_a, \sigma_b, k, f_{-})$ is admissible, then 
  \begin{equation}
    \lim_{\delta\to 0}\lim_{\gamma\to 0}\lim_{\eps\to 0}\int_{\bbS^{n-1}}\frac{u^{\eps, \delta}(x, \theta)}{\delta} h\left(\frac{\theta - \theta'}{\gamma}\right)d\theta = \exp \left( \int_0^{\tau_{-}(x,{\theta'})} \Sigma(|\aver{w}|)(x-s{\theta'})ds\right),
  \end{equation}
  where $u^{\eps,\delta}$ is the unique solution to~\eqref{EQ: TP RTE} with boundary condition $f_{-}$ and $w$ is the unique solution to~\eqref{EQ: TP RTE} with the boundary condition $f_{-} = c_0$.
\end{theorem}
\begin{proof}
  Let $w(x, \theta)$ be the solution to the following equation
  \begin{equation}\label{EQ: W'}
    \begin{aligned}
      ( T + \Sigma(|\aver{w}|) + K ) w &= 0 &\text{ in } &\Omega \times \mathbb{S}^{n-1}, \\
      w(x, \theta) &= c_0 \, &\text{ on }& \Gamma_{-}.
    \end{aligned}
  \end{equation}
  Then $w\in L^{\infty}(\Omega\times\bbS^{n-1})$, which implies 
  \begin{equation}\label{EQ: PART I}
    \lim_{\delta\to 0}\lim_{\gamma\to 0}\lim_{\eps\to 0}\int_{\bbS^{n-1}} \frac{w(x, \theta)}{\delta} h\left(\frac{\theta - \theta'}{\gamma}\right) d\theta = 0.
  \end{equation}
  We  denote  $\phi = \frac{1}{\delta}\left( u^{\eps,\delta}-  w\right)$. It satisfies 
\begin{equation}
  \begin{aligned}
    (T + \Sigma(|\aver{u^{\eps,\delta}}|) + K) \phi &= \sigma_b \left(\frac{|\aver{u^{\eps,\delta}}| - |\aver{w}|}{\delta}\right)w &\text{ in } &\Omega \times \mathbb{S}^{n-1}, \\
    \phi(x, \theta) &=\frac{1}{\omega_{n-1}\eps^{n-1}} h \left(\frac{\theta - \theta'}{\eps}\right) \, &\text{ on }& \Gamma_{-}.
  \end{aligned}
\end{equation}
Therefore, the solution $\phi$ can be written in the following form,
\begin{equation}\nonumber
  \begin{aligned}
    \phi(x, \theta) &=   \exp \left( \int_0^{\tau_{-}(x,\theta)} \Sigma(|\aver{u^{\eps, \delta}}|)(x-s\theta)ds\right) \frac{1}{\omega_{n-1}\eps^{n-1}} h \left(\frac{\theta - \theta'}{\eps}\right) \\
     &\quad + \int_{0}^{\tau_{-}(x, \theta)} \exp \left( \int_0^l \Sigma(|\aver{u^{\eps, \delta}}|)(x-s\theta)ds\right) K\phi(x- l\theta, \theta) dl \\
    &\quad -\int_{0}^{\tau_{-}(x, \theta)} \exp \left( \int_0^l \Sigma(|\aver{u^{\eps, \delta}}|)(x-s\theta)ds\right)\left[\sigma_b \left(\frac{|\aver{u^{\eps,\delta}}| - |\aver{w}|}{\delta}\right)w(x-l\theta, \theta)\right]  dl\\
    &=\cL_1(x, \theta) + \cL_2(x, \theta) + \cL_3(x, \theta).
  \end{aligned}
\end{equation}
Integrate $\phi(x, \theta)$ over $\bbS^{n-1}$ and note $||\aver{u^{\eps,\delta}}| - |\aver{w}||\le \delta |\aver{\phi}|$, to obtain 
\begin{equation}
  \|\aver{\phi}\|_{L^{\infty}(\Omega)} \le \frac{1}{(1-\mu)(1- \nu)}\left|\int_{\bbS^{n-1}} \frac{1}{\omega_{n-1}\eps^{n-1}} h \left(\frac{\theta - \theta'}{\eps}\right) d\theta \right|.
\end{equation}
This implies that $\cL_2, \cL_3$ are both uniformly bounded in $L^{\infty}(\Omega\times\bbS^{n-1})$, hence 
\begin{equation}\label{EQ: PART II}
  \begin{aligned}
&\lim_{\delta\to 0}\lim_{\gamma\to 0}\lim_{\eps\to 0}\int_{\bbS^{n-1}} \phi(x, \theta) h\left(\frac{\theta - \theta'}{\gamma}\right)d\theta \\=& \lim_{\delta\to 0}\lim_{\gamma\to 0}\lim_{\eps\to 0}\int_{\bbS^{n-1}}\cL_1(x, \theta)\frac{1}{\omega_{n-1}\eps^{n-1}} h \left(\frac{\theta - \theta'}{\eps}\right) h\left(\frac{\theta - \theta'}{\gamma}\right)d\theta \\
=&\lim_{\delta\to 0} \exp \left( \int_0^{\tau_{-}(x,{\theta'})} \Sigma(|\aver{u^{\eps, {\theta'}}}|)(x-s{\theta'})ds\right) \\=& \exp \left( \int_0^{\tau_{-}(x,{\theta'})} \Sigma(|\aver{w}|)(x-s{\theta'})ds\right).
  \end{aligned}
\end{equation}

Combine this with~\eqref{EQ: PART I} to obtain  
\begin{equation}
  \lim_{\delta\to 0}\lim_{\gamma\to 0}\lim_{\eps\to 0}\int_{\bbS^{n-1}}\frac{u^{\eps, \delta}(x, \theta)}{\delta} h\left(\frac{\theta - \theta'}{\gamma}\right)d\theta = \exp \left( \int_0^{\tau_{-}(x,{\theta'})} \Sigma(|\aver{w}|)(x-s{\theta'})ds\right).
\end{equation}
\end{proof}

Theorem~\ref{THM: B} implies that $\Sigma(|w|)$ can be reconstructed from the albedo operator. Therefore the solution $w$ of~\eqref{EQ: W'} can be uniquely determined {and there exists a constant $C > 0$ such that $w(x, \theta)\ge C c_0$ by Lemma~\ref{LEM: UNIQUE}}; and when $\sigma_a$ is  known, one can find $\sigma_b = (\Sigma(|w|) - \sigma_a )/|\aver{w}|$.  

\section{Scattering-free media}\label{SEC: SCATTER FREE}
For  media with $k = 0$, there exists a more direct explicit reconstruction method. Moreover, no smallness assumptions on the boundary source are needed. Equation \eqref{EQ: TP RTE} reduces to 
 \begin{equation}\label{transport_eq}
 \theta\cdot\nabla u +\sigma_a  u +\sigma_b \langle u \rangle u =0.
 \end{equation} 
Choose the boundary condition
 \begin{equation}\label{2}
 f_-= v_- (x)\delta_{\theta_0}(\theta)
 \end{equation}
in \eqref{EQ: TP RTE}  with some  $v_-(x)\ge0$ in $C^1$.  We are going to look for a non-negative weak solution, i.e., for a solution of the integrated equation
\begin{equation}\label{EQ:P1}
    u(x, \theta) = v_{-}(x - \tau_{-}(x, \theta)\theta) \delta_{\theta_0}(\theta) \exp\left( -\int_0^{\tau_{-}(x,\theta)} (\sigma_a + \sigma_b \aver u)(x-s\theta)ds \right) 
\end{equation}
in the following class: $u(x,\theta)$ is a measure-valued function in $\theta$,    $C^1(\Omega)\cap C(\bar\Omega)$ in the $x$ variable. Then $\aver u (x)$ is in the latter space.   By \eqref{EQ:P1}, $u = \delta_{\theta_0}(\theta)v$ with $v\in C(\bar\Omega\times \bbS^{n-1})$; and also, $v$ is $C^1$ except for $(x,\theta)$ such that $x\in\partial\Omega$ and $\theta$ is tangent to $\partial\Omega$ (which is $\partial\Gamma_0$). Clearly, only the value of $v$ at $\theta=\theta_0$ matters for $u$. With some abuse of notation, we denote $v(x,\theta_0)$ by $v(x)$. Then by \eqref{2}, $v$ must satisfy the boundary condition $v=v_-$ on $\partial\Omega$.

 
 In view of the $C^1$ regularity of $v$ as stated above, we can differentiate \eqref{EQ:P1} to get back to the differential form \eqref{transport_eq}, which in this case reduces to 
 \begin{equation}\label{5}
 \left(\theta_0\cdot\nabla v+ \sigma_a v+\sigma_b v^2\right) =0, 
 \end{equation}
 since $\langle u\rangle=v$. 
 Here, $\sigma_a$ and $\sigma_b$ can depend on $\theta$ as well; then $\theta=\theta_0$ above.
Therefore, on each line $s\mapsto (x_0+s\theta_0,\theta_0)$, the equation reduces to
 \begin{equation}\label{6}
 v'+\sigma_a v+\sigma_bv^2=0.
 \end{equation}
 This is a homogeneous Riccati equation. For each initial condition $v(0)=v_-(x_0)$, we measure $v(\tau_+(x,\theta_0))$. 

 Let $\mu(t)=\exp\big(-\int_0^t\sigma_a(s)\,d s\big)$; then $1/\mu$ is the integrating factor. Multiply \eqref{6} by $1/\mu$ to get
 \begin{equation}\label{6a}
 (v/\mu)'+\sigma_b  v^2/\mu=0.
 \end{equation}
 This is a separable ODE for $v/\mu$ and the solution satisfies
 \begin{equation}\label{6b}
 \frac{\mu}{ v} = \frac1{v_-(x_0)}+\int_0^s \mu(t) \sigma_b(t) \,d t,
 \end{equation}
 therefore,
 \begin{equation}\label{6c}
 v(s) = \mu(s)\left(\frac1{v_-(x_0)}+ \int_0^s \mu(t) \sigma_b(t) \,d t\right)^{-1}.
 \end{equation}
 Hence, at $s=\tau_+(x_0,\theta_0)$ we recover the attenuated X-ray transform of $\sigma_b$ with attenuation $\sigma_a$, assuming $\sigma_a$ known. One way to recover $\sigma_a$ is to replace  $v_-(x_0)$ by $\delta v_-(x_0)$ as in the previous section with $\delta\to0$, then we get the X-ray transform $-\log \mu(\tau_{+}(x,\theta_0))$ of $\sigma_a$; and by varying $\theta$, we can recover $\sigma_a$. Then we recover $\sigma_b$ by inverting the attenuated X-ray transform of $\sigma_b$, see \cite{Buk-Kaz, Novikov}.  

If we do not want to deal with small signals which may be corrupted by background noise too much, we can proceed as following. To reconstruct $\sigma_a$,  we choose two distinct boundary sources $f_{-,j} = v_{-,j}(x)\delta_{\theta_0}(\theta)$, $j=1,2$ such that $\forall x\in\partial\Omega$, $v_{-,1}(x) >  v_{-,2}(x)$. Let $v_1, v_2$ be the solutions to~\eqref{6} with $v_j(0) = v_{-,j}(x_0)$, then from~\eqref{6c} we observe
\begin{equation}
  \frac{1}{v_j(s)} = \frac{1}{\mu(s)} \left(\frac1{v_{-,j}(x_0)}+ \int_0^s \mu(t) \sigma_b(t) \,d t\right),\quad j = 1,2.
\end{equation}
Subtracting the above formulas with $j=1, 2$, we obtain
\begin{equation}
  \frac{1}{v_1(s)} - \frac{1}{v_2(s)} = \frac{1}{\mu(s)} \left( \frac1{v_{-,1}(x_0)} - \frac1{v_{-,2}(x_0)}  \right),
\end{equation} 
which implies 
\begin{equation}\label{78}
  \mu(s) = \left(   \frac{1}{v_1(s)} - \frac{1}{v_2(s)}    \right)^{-1}  \left( \frac1{v_{-,1}(x_0)} - \frac1{v_{-,2}(x_0)}  \right).
\end{equation}
Take $s=\tau_+(x_0,\theta_0)$ to get $\mu(\tau_+(x_0,\theta_0)) = \exp(-X\sigma_a(x_0, \theta_0))$, where $X$ is the X-ray transform, can be determined by \eqref{78}.  Therefore,  we can recover $\sigma_a$ first by  varying $\theta_0$ and inverting the X-ray transform of $\sigma_a$ as above. After that, we  recover $\sigma_b$ as above. 

Also, one can take $v_-(x_0)$ approximating $\delta_{x_0}(x)$, this corresponds to a single beam. 

Therefore, we proved the following.

\begin{theorem}\label{thm_sc-free}
Assume $k=0$. Let $\sigma_a$ and $\sigma_b$ depend on $x$ only and be in  $C^0(\overline{\Omega})$.  Then $\mathcal{A}$ acting on $f_-$ as in \eqref{2},  determines $\sigma_a$, $\sigma_b$  uniquely by inverting their attenuated, respectively the non-attenuated X-ray transforms, which can be determined by \eqref{6c} and \eqref{78}.  
\end{theorem}


\begin{thebibliography}{10}

\bibitem{Bal-2009}
G.~Bal.
\newblock Inverse transport theory and applications.
\newblock {\em Inverse Problems}, 25(5):053001, 48, 2009.

\bibitem{Bal_J_stability}
G.~Bal and A.~Jollivet.
\newblock Stability estimates in stationary inverse transport.
\newblock {\em Inverse Probl. Imaging}, 2(4):427--454, 2008.

\bibitem{Bal_J_stab2}
G.~Bal and A.~Jollivet.
\newblock Generalized stability estimates in inverse transport theory.
\newblock {\em Inverse Probl. Imaging}, 12(1):59--90, 2018.

\bibitem{Bal-Jol-2011}
G.~Bal, A.~Jollivet, I.~Langmore, and F.~Monard.
\newblock Angular average of time-harmonic transport solutions.
\newblock {\em Comm. Partial Differential Equations}, 36(6):1044--1070, 2011.

\bibitem{Bal-K-M-2008}
G.~Bal, I.~Langmore, and F.~Monard.
\newblock Inverse transport with isotropic sources and angularly averaged
  measurements.
\newblock {\em Inverse Probl. Imaging}, 2(1):23--42, 2008.

\bibitem{Bal-Monard-2012}
G.~Bal and F.~Monard.
\newblock Inverse transport with isotropic time-harmonic sources.
\newblock {\em SIAM J. Math. Anal.}, 44(1):134--161, 2012.

\bibitem{bardsley2018quantitative}
P.~Bardsley, K.~Ren, and R.~Zhang.
\newblock Quantitative photoacoustic imaging of two-photon absorption.
\newblock {\em Journal of biomedical optics}, 23(1):016002, 2018.

\bibitem{Buk-Kaz}
A.~A. Bukhgeim and S.~G. Kazantsev.
\newblock Inversion formula for the {F}an-{B}eam attenuated {R}adon transform
  in a unit disk.
\newblock {\em Sobolev Institute of Mathematics, Siberian Branch of Russian
  Acad. Sci., Novosibirsk}, preprint No. 99, 2002.

\bibitem{Ch-St-Osaka}
M.~Choulli and P.~Stefanov.
\newblock An inverse boundary value problem for the stationary transport
  equation.
\newblock {\em Osaka J. Math.}, 36(1):87--104, 1999.

\bibitem{dautray2012mathematical}
R.~Dautray and J.-L. Lions.
\newblock {\em Mathematical analysis and numerical methods for science and
  technology: Volume 6 Evolution Problems II}.
\newblock Springer Science \& Business Media, 2012.

\bibitem{Klingenberg-Lai-Li-2021}
C.~Klingenberg, R.-Y. Lai, and Q.~Li.
\newblock Reconstruction of the emission coefficient in the nonlinear radiative
  transfer equation.
\newblock {\em SIAM J. Appl. Math.}, 81(1):91--106, 2021.

\bibitem{lai2019parameter}
R.-Y. Lai and Q.~Li.
\newblock Parameter reconstruction for general transport equation.
\newblock {\em arXiv:1904.10049}, 2019.

\bibitem{lai2019inverse}
R.-Y. Lai, Q.~Li, and G.~Uhlmann.
\newblock Inverse problems for the stationary transport equation in the
  diffusion scaling.
\newblock {\em SIAM Journal on Applied Mathematics}, 79(6):2340--2358, 2019.

\bibitem{lai2021reconstruction}
R.-Y. Lai, G.~Uhlmann, and Y.~Yang.
\newblock Reconstruction of the collision kernel in the nonlinear {B}oltzmann
  equation.
\newblock {\em SIAM J. Appl. Math.}, 53(1):1049--1069, 2021.

\bibitem{Makarov:08}
N.~S. Makarov, M.~Drobizhev, and A.~Rebane.
\newblock Two-photon absorption standards in the 550--1600 nm excitation
  wavelength range.
\newblock {\em Opt. Express}, 16(6):4029--4047, Mar 2008.

\bibitem{SMT-gauge}
S.~McDowall, P.~Stefanov, and A.~Tamasan.
\newblock Gauge equivalence in stationary radiative transport through media
  with varying index of refraction.
\newblock {\em Inverse Probl. Imaging}, 4(1):151--167, 2010.

\bibitem{SMT-stab}
S.~McDowall, P.~Stefanov, and A.~Tamasan.
\newblock Stability of the gauge equivalent classes in inverse stationary
  transport.
\newblock {\em Inverse Problems}, 26(2):025006, 19, 2010.

\bibitem{MST_2011}
S.~McDowall, P.~Stefanov, and A.~Tamasan.
\newblock Stability of the gauge equivalent classes in inverse stationary
  transport in refractive media.
\newblock {\em Contemporary Math.}, 559:85--100, 2011.

\bibitem{Steve2004}
S.~R. McDowall.
\newblock An inverse problem for the transport equation in the presence of a
  {R}iemannian metric.
\newblock {\em Pacific J. Math.}, 216(2):303--326, 2004.

\bibitem{Steve2D}
S.~R. McDowall.
\newblock Optical tomography on simple {R}iemannian surfaces.
\newblock {\em Comm. Partial Differential Equations}, 30(7-9):1379--1400, 2005.

\bibitem{Novikov}
R.~G. Novikov.
\newblock An inversion formula for the attenuated {X}-ray transformation.
\newblock {\em Ark. Mat.}, 40(1):145--167, 2002.

\bibitem{pawlicki2009two}
M.~Pawlicki, H.~A. Collins, R.~G. Denning, and H.~L. Anderson.
\newblock Two-photon absorption and the design of two-photon dyes.
\newblock {\em Angewandte Chemie International Edition}, 48(18):3244--3266,
  2009.

\bibitem{Ren_Zhong2020}
K.~Ren and Y.~Zhong.
\newblock Unique determination of absorption coefficients in a semilinear
  transport equation.
\newblock {\em arXiv:2007.09516}, 2020.

\bibitem{rumi2010two}
M.~Rumi and J.~W. Perry.
\newblock Two-photon absorption: an overview of measurements and principles.
\newblock {\em Advances in Optics and Photonics}, 2(4):451--518, 2010.

\bibitem{Sh-transport}
V.~A. Sharafutdinov.
\newblock Inverse problem of determining a source in the stationary transport
  equation on a {R}iemannian manifold.
\newblock {\em Zap. Nauchn. Sem. S.-Peterburg. Otdel. Mat. Inst. Steklov.
  (POMI)}, 239(Mat. Vopr. Teor. Rasprostr. Voln. 26):236--242, 270, 1997.

\bibitem{S-inside-out}
P.~Stefanov.
\newblock Inverse problems in transport theory.
\newblock In {\em Inside out: inverse problems and applications}, volume~47 of
  {\em Math. Sci. Res. Inst. Publ.}, pages 111--131. Cambridge Univ. Press,
  Cambridge, 2003.

\bibitem{ST-PAMC}
P.~Stefanov and A.~Tamasan.
\newblock Uniqueness and non-uniqueness in inverse radiative transfer.
\newblock {\em Proc. Amer. Math. Soc.}, 137(7):2335--2344, 2009.

\bibitem{SU-optical2D}
P.~Stefanov and G.~Uhlmann.
\newblock Optical tomography in two dimensions.
\newblock {\em Methods Appl. Anal.}, 10(1):1--9, 2003.

\bibitem{SU-APDE}
P.~Stefanov and G.~Uhlmann.
\newblock An inverse source problem in optical molecular imaging.
\newblock {\em Anal. PDE}, 1(1):115--126, 2008.

\bibitem{van1985two}
E.~W. Van~Stryland, H.~Vanherzeele, M.~A. Woodall, M.~Soileau, A.~L. Smirl,
  S.~Guha, and T.~F. Boggess.
\newblock Two photon absorption, nonlinear refraction, and optical limiting in
  semiconductors.
\newblock {\em Optical Engineering}, 24(4):244613, 1985.

\bibitem{zhao2019instability}
H.~Zhao and Y.~Zhong.
\newblock Instability of an inverse problem for the stationary radiative
  transport near the diffusion limit.
\newblock {\em SIAM Journal on Mathematical Analysis}, 51(5):3750--3768, 2019.

\end{thebibliography}


\end{document}